\documentclass[12pt]{article}
\usepackage{amssymb}
\usepackage{amsthm}
\usepackage{amsmath}
\usepackage{graphicx}
\usepackage{enumerate}
\usepackage{pict2e}
\usepackage{framed}

\DeclareMathOperator{\Max}{Max}

\newtheorem{theorem}{Theorem}[section]
\newtheorem{definition}[theorem]{Definition}
\newtheorem{lemma}[theorem]{Lemma}
\newtheorem{proposition}[theorem]{Proposition}

\newtheorem{remark}[theorem]{Remark}
\newtheorem{example}[theorem]{Example}
\newtheorem{corollary}[theorem]{Corollary}
\title{Special filters in bounded lattices}
\author{Ivan~Chajda, Miroslav~Kola\v r\'ik and Helmut~L\"anger}
\date{}

\begin{document}

\footnotetext{Support of the research of the first author by IGA, project P\v rF~2023~010, and of the third author by the Austrian Science Fund (FWF), project I~4579-N, entitled ``The many facets of orthomodularity'', is gratefully acknowledged.}

\maketitle

\begin{abstract}
M.S.\ Rao recently investigated some sorts of special filters in distributive pseudocomplemented lattices. In our paper we extend this study to lattices which need neither be distributive nor pseudocomplemented. For this sake we define a certain modification of the notion of a pseudocomplement as the set of all maximal elements belonging to the annihilator of the corresponding element. We prove several basic properties of this notion and then define coherent, closed and median filters as well as $D$-filters. In order to be able to obtain valuable results we often must add some additional assumptions on the underlying lattice, e.g.\ that this lattice is Stonean or $D$-Stonean. Our results relate properties of lattices and of corresponding filters. We show how the structure of a lattice influences the form of its filters and vice versa.
\end{abstract}



\section{Introduction and basic concepts}

M.S.~Rao recently studied various types of filters and ideals in distributive lattices with pseudocomplements, see \cite{R22} and \cite{Ra}. Some of these findings require the lattice under consideration to be either Stone or weakly Stone, as indicated in the referenced papers. There is the question whether the distributivity assumption can be relaxed and, more specifically, if the pseudocomplementation assumption can be eliminated. This approach, in which we consider instead of the pseudocomplement a certain subset (in fact an antichain) of elements behaving similar to the pseudocomplement, has already been introduced by I.~Chajda and H.~L\"anger in \cite{CL}. In this paper the authors show that this construction can be successfully used in order to introduce the implication and negation connectives in an intuitionistic-like logic based on any bounded lattice that satisfies the Ascending Chain Condition.

In this paper we will deal with bounded lattices $\mathbf L=(L,\vee,\wedge,0,1)$ satisfying the Ascending Chain Condition (ACC). We identify singletons with their unique element, i.e.\ we will write $a$ instead of $\{a\}$. For $a\in L$ and for subsets $A,B$ of $L$ we define
\begin{align*}
  A\vee B & :=\{x\vee y\mid x\in A\text{ and }y\in B\}, \\
A\wedge B & :=\{x\wedge y\mid x\in A\text{ and }y\in B\}, \\
      A^0 & :=\Max\{x\in L\mid x\wedge y=0\text{ for all }y\in A\}, \\
      a^0 & :=\Max\{x\in L\mid x\wedge a=0\}.
\end{align*}
Here and in the following $\Max A$ denotes the set of all maximal elements of $A$ which obviously is an antichain. Observe that $0^0=1$ and $1^0=0$ and that $A^0\neq\emptyset$. The set $a^0$ is in fact a generalization of the {\em pseudocomplement} of $a$ introduced by O.~Frink \cite F or a modification of the annihilator since the set
\[
\{x\in L\mid x\wedge y=0\text{ for all }y\in A\}
\]
is in fact the {\em annihilator} of the set $A$ as known in lattice theory. (The pseudocomplement of $a$ is the greatest element of $\{x\in L\mid x\wedge a=0\}$.) The advantage of our approach is that $a^0$ can be defined in any lattice with $0$. Of course, we must pay for this advantage by the fact that $a^0$ need not be an element of $L$, but may be a subset of $L$, namely an antichain of $\mathbf L$.

It is easy to see that in a bounded completely distributive lattice every element has a pseudocomplement.

The element $a$ is called {\em dense} if $a^0=0$ and {\em sharp} if $a^{00}=a$. Let $D$ ($S$) denote the set of all dense (sharp) elements of $L$. Clearly the following holds:
\begin{align*}
& a\in D\text{ if and only if }a^{00}=1, \\
& 1\in D, \\
& 0,1\in S, \\
& D\cap S=1.
\end{align*}
Moreover, we define the following binary relations and unary operators on $2^L$:
\begin{align*}
    A\leq B & \text{ if }x\leq y\text{ for all }x\in A\text{ and all }y\in B, \\
  A\leq_1 B & \text{ if for every }x\in A\text{ there exists some }y\in B\text{ with }x\leq y, \\
  A\leq_2 B & \text{ if for every }y\in B\text{ there exists some }x\in A\text{ with }x\leq y, \\
      A=_1B & \text{ if both }A\leq_1 B\text{ and }B\leq_1 A, \\
\overline A & :=\{x\in L\mid1\in x^{00}\vee y^{00}\text{ for each }y\in A\}, \\
        A^D & :=\{x\in L\mid x\vee y\in D\text{ for all }y\in A\}.
\end{align*}
The set $A$ is called {\em closed} if $\overline{\overline A}=A$. Observe that $\overline0=D$ and $\overline1=L$.

\begin{lemma}\label{lem4}
Let $(L,\vee,\wedge,0,1)$ be a bounded lattice satisfying the {\rm ACC}. Then $\overline D=L$ and $\overline L=D$.
\end{lemma}

\begin{proof}
If $a\in L$ then $1\in a^{00}\vee x^{00}$ for all $x\in D$ and hence $a\in\overline D$. This shows $L\subseteq\overline D$, i.e.\ $\overline D=L$. Now let $a\in\overline L$. Then $1\in a^{00}\vee x^{00}$ for all $x\in L$, especially $1\in a^{00}\vee0^{00}=a^{00}$. Since $a^{00}$ is an antichain we have $a^{00}=1$ whence $a\in D$. This shows $\overline L\subseteq D$. Conversely, if $a\in D$ then $a^{00}=1$ and hence $1\in a^{00}\vee x^{00}$ for all $x\in L$, i.e.\ $a\in\overline L$. This shows $D\subseteq\overline L$. Altogether, we obtain $\overline L=D$.
\end{proof}

\begin{remark}
The pair $(A\mapsto\overline A,A\mapsto\overline A)$ is the Galois-correspondence between $(2^L,\subseteq)$ and $(2^L,\subseteq)$ induced by the binary relation $\{(x,y)\in L^2\mid1\in x^{00}\vee y^{00}\}$ on $L$. Hence $A\mapsto\overline{\overline A}$ is a closure operator on $(2^L,\subseteq)$ and therefore the set of all closed subsets of $L$ forms a closure system of $L$, i.e.\ it is closed under arbitrary intersections. This means that the closed subsets of $L$ form a complete lattice with respect to inclusion with smallest element $\overline L=D$ and greatest element $\overline\emptyset=L$ and we have the following properties:
\begin{align*}
                                A & \subseteq B\text{ implies }\overline B\subseteq\overline A, \\
                                A & \subseteq\overline{\overline A}, \\
\overline{\overline{\overline A}} & =\overline A, \\
                                A & \subseteq\overline B\text{ if and only if }B\subseteq\overline A.
\end{align*}
Analogously, the pair $(A\mapsto A^D,A\mapsto A^D)$ is the Galois-correspondence between $(2^L,\subseteq)$ and $(2^L,\subseteq)$ induced by the binary relation $\{(x,y)\in L^2\mid x\vee y\in D\}$ on $L$.
\end{remark}

Recall that a non-empty subset $F$ of $L$ is called a {\em filter} of $\mathbf L$ if $x\wedge y,x\vee z\in F$ for all $x,y\in F$ and all $z\in L$. Observe that for every $x\in L$ the set $F_x:=\{y\in L\mid x\leq y\}$ is a filter of $\mathbf L$, the so-called {\em principal filter} generated by $x$. A filter $F$ of $\mathbf L$ is called a {\em $D$-filter} if $D\subseteq F$. It is elementary that every filter of a finite lattice is principal.

\section{Filters and $D$-filters}

It is evident that if a given bounded lattice $\mathbf L=(L,\vee,\wedge,0,1)$ is pseudocomplemented then for each element $x\in L$ we have $x^0=x^*$. We start this section with emphasizing certain differences between the properties of our concept $x^0$ and pseudocomplements.

It is well-known that for distributive pseudocomplemented lattices $(L,\vee,\wedge,0)$ with bottom element $0$ and $a,b\in L$ the following holds:
\begin{enumerate}[{\rm(i)}]
\item $a\leq b$ implies $b^*\leq a^*$,
\item $a\leq a^{**}$ and $a^{***}=a^*$,
\item $(a\vee b)^*=a^*\wedge b^*$,
\item $(a\wedge b)^{**}=a^{**}\wedge b^{**}$.
\end{enumerate}

The following example shows that not all being valid for pseudocomplements, also holds for $x^0$.

\begin{example}
Consider the non-distributive lattice $\mathbf L$ depicted in Fig.~1:

\vspace*{-4mm}

\begin{center}
\setlength{\unitlength}{1.5mm}
\begin{picture}(75,45)
\put(10,20){\circle*{1.4}}
\put(20,10){\circle*{1.4}}
\put(20,30){\circle*{1.4}}
\put(30,20){\circle*{1.4}}
\put(40,30){\circle*{1.4}}
\put(50,20){\circle*{1.4}}
\put(50,40){\circle*{1.4}}
\put(60,30){\circle*{1.4}}
\put(70,30){\circle*{1.4}}
\put(20,10){\line(-1,1){10}}
\put(20,10){\line(1,1){10}}
\put(10,20){\line(1,1){10}}
\put(30,20){\line(-1,1){10}}
\put(20,10){\line(3,1){30}}
\put(30,20){\line(3,1){30}}
\put(10,20){\line(3,1){30}}
\put(20,30){\line(3,1){30}}
\put(50,20){\line(-1,1){10}}
\put(50,20){\line(1,1){10}}
\put(50,20){\line(2,1){20}}
\put(40,30){\line(1,1){10}}
\put(60,30){\line(-1,1){10}}
\put(70,30){\line(-2,1){20}}
\put(19.2,6.9){$0$}
\put(7.5,19){$a$}
\put(27.5,19){$b$}
\put(17.5,29){$c$}
\put(51,18.5){$d$}
\put(41.2,29.3){$e$}
\put(60.9,29){$f$}
\put(71,29){$g$}
\put(49.2,41.2){$1$}
\put(35,2){{\rm Fig.~1}}
\put(10,-2){{\rm Non-distributive non-pseudocomplemented lattice}}
\end{picture}
\end{center}
We have
\[
\begin{array}{r|r|r|r|r|r|r|r|r|r}
          x &   0 &   a &   b &   c &     d &     e &        f &     g &   1 \\
\hline
        x^0 &   1 &  fg &  eg &   g &     c &     c &        a &     c &   0 \\
\hline
     x^{00} &   0 &   a &   b &   c &     g &     g &       fg &     g & F_0 \\
\hline
\overline x & F_1 & F_d & F_d & F_d & abcf1 & abcf1 & abcdefg1 & abcf1 & L
\end{array}
\]
$D=F_1$ and $S:=\{0,b,c,g,1\}$. {\rm(}Here and in the following we often write $fg$ instead of $\{f,g\}$, and so on.{\rm)} Hence $\mathbf L$ is not pseudocomplemented. In $\mathbf L$ we have:
\begin{enumerate}[{\rm(i)}]
\item $a\leq c$ and $c^0=g\nleq\{f,g\}=a^0$, but $c^0\leq_1 a^0$.
\item $f\not\leq\{f,g\}=f^{00}$, but $f\leq_1f^{00}$.
\item $(a\vee b)^0=c^0=g\neq\{d,g\}=\{f,g\}\wedge\{e,g\}=a^0\wedge b^0$, but $(a\vee b)^0\leq_1a^0\wedge b^0$.
\item $(d\wedge f)^{00}=d^{00}=g\neq\{d,g\}=g\wedge\{f,g\}=d^{00}\wedge f^{00}$, but $(d\wedge f)^{00}=_1d^{00}\wedge f^{00}$.
\end{enumerate}
\end{example}

On the other hand, some of the properties of pseudocomplements are preserved also here, see the following result.

\begin{theorem}\label{th2}
Let $(L,\vee,\wedge,0,1)$ be a bounded lattice satisfying the {\rm ACC}, $x,y\in L$ and $A,B\subseteq L$. Then the following holds:
\begin{enumerate}[{\rm(i)}]
\item $a\wedge b=0$ for all $a\in A$ and all $b\in A^0$, especially, $x\wedge x^0=0$,
\item If $a\wedge b=0$ for all $a\in A$ and all $b\in B$ then $A\leq_1B^0$, especially, $x\wedge y=0$ implies $x\leq_1y^0$,
\item $A\leq_1A^{00}$, especially, $x\leq_1x^{00}$,
\item $A\leq_1B$ implies $B^0\leq_1A^0$, especially, $x\leq y$ implies $y^0\leq_1x^0$,
\item $A\leq_1B^0$ if and only if $B\leq_1A^0$, especially, $x\leq_1y^0$ if and only if $y\leq_1x^0$,
\item $a=_1B$ implies $a=B$,
\item $A^0=_1B^0$ implies $A^0=B^0$,
\item $A^{000}=A^0$, especially, $x^{000}=x^0$,
\item $(A\vee B)^0\leq_1A^0\wedge B^0$, especially, $(x\vee y)^0\leq_1x^0\wedge y^0$,
\item $(A\wedge B)^{00}=_1A^{00}\wedge B^{00}$, especially, $(x\wedge y)^{00}=_1x^{00}\wedge y^{00}$,
\item $A^{00}\vee B^{00}\leq_1(A^0\wedge B^0)^0$ implies $(A^{00}\vee B^{00})^0=_1A^0\wedge B^0$, especially, $x^{00}\vee y^{00}\leq_1(x^0\wedge y^0)^0$ implies $(x^{00}\vee y^{00})^0=_1x^0\wedge y^0$.
\end{enumerate}
\end{theorem}

\begin{proof}
\
\begin{enumerate}[(i)]
\item If $b\in A\wedge A^0$ then there exists some $c\in A$ and some $d\in A^0$ with $c\wedge d=b$. Since $x\wedge d=0$ for all $x\in A$, we have $b=c\wedge d=0$.
\item If $A\wedge B=0$ and $b\in A$ then $b\wedge x=0$ for all $x\in B$ and hence there exists some $c\in B^0$ with $b\leq c$.
\item This follows from (i) and (ii).
\item If $A\leq_1B$ then
\[
\{x\in L\mid x\wedge B=0\}\subseteq\{x\in L\mid x\wedge A=0\}
\]
and hence $B^0\leq_1A^0$.
\item If $A\leq_1B^0$ then $B\leq_1B^{00}\leq_1A^0$ by (iii) and (iv). Analogously, $B\leq_1A^0$ implies $A\leq_1B^0$.
\item Let $b\in B^0$. Because of $B^0\leq_1a$ we have $b\leq a$, and because of $a\leq_1B^0$ there exists some $c\in B^0$ with $a\leq c$. Together we obtain $b\leq a\leq c$ and hence $b\leq c$. Since both $b$ and $c$ belong to the antichain $B^0$ we conclude $b=c$ and hence $b=a$ showing $B^0=a$.
\item Assume $b\in A^0$. Because of $A^0\leq_1B^0$ there exists some $c\in B^0$ with $b\leq c$, and because of $B^0\leq_1A^0$ there exists some $d\in A^0$ with $c\leq d$. Together we obtain $b\leq c\leq d$ and hence $b\leq d$. Since both $b$ and $d$ belong to the antichain $A^0$ we conclude $b=d$ and hence $b=c\in B^0$. This shows $A^0\subseteq B^0$. Interchanging the roles of $A^0$ and $B^0$ gives $B^0\subseteq A^0$. Together we obtain $A^0=B^0$.
\item Because of (iii) we have $A\leq_1A^{00}$ from which we conclude $A^{000}=(A^{00})^0\leq_1A^0$ by (iv). Again by (iii) we have $A^0\leq_1(A^0)^{00}=A^{000}$. Altogether, $A^{000}=_1A^0$. Applying (vii) yields $A^{000}=A^0$.
\item We have $A,B\leq_1A\vee B$ and hence $(A\vee B)^0\leq_1A^0,B^0$ by (iv) which implies $(A\vee B)^0\leq_1A^0\wedge B^0$.
\item We have $A\wedge B\leq_1A^{00}\wedge B^{00}$ according to (iii) and hence $(A\wedge B)^{00}\leq_1(A^{00}\wedge B^{00})^{00}$ according to (iv). Using (ii), (i) and (iv) we see that any of the following statements implies the next one:
\begin{align*}
(A\wedge B)^0\wedge A^{00}\wedge B^{00}\wedge A\wedge B & \subseteq\{0\}, \\
        (A\wedge B)^0\wedge A^{00}\wedge B^{00}\wedge A & \leq_1B^0, \\
        (A\wedge B)^0\wedge A^{00}\wedge B^{00}\wedge A & \leq_1B^0\wedge B^{00}, \\
        (A\wedge B)^0\wedge A^{00}\wedge B^{00}\wedge A & \subseteq\{0\}, \\
                (A\wedge B)^0\wedge A^{00}\wedge B^{00} & \leq_1A^0, \\
                (A\wedge B)^0\wedge A^{00}\wedge B^{00} & \leq_1A^0\wedge A^{00}, \\
                (A\wedge B)^0\wedge A^{00}\wedge B^{00} & =0, \\
                                          (A\wedge B)^0 & \leq_1(A^{00}\wedge B^{00})^0, \\
                             (A^{00}\wedge B^{00})^{00} & \leq_1(A\wedge B)^{00}.
\end{align*}
Together we obtain $(A\wedge B)^{00}=_1(A^{00}\wedge B^{00})^{00}$ and hence $(A\wedge B)^{00}=(A^{00}\wedge B^{00})^{00}$ according to (vii). Now we have $A^{00}\wedge B^{00}\leq_1A^{00},B^{00}$ and hence $(A^{00}\wedge B^{00})^{00}\leq_1A^{00},B^{00}$ by (iv) and (viii) which implies $(A^{00}\wedge B^{00})^{00}\leq_1A^{00}\wedge B^{00}$. Together with $A^{00}\wedge B^{00}\leq_1(A^{00}\wedge B^{00})^{00}$ that follows from (iii) this yields $(A\wedge B)^{00}=_1(A^{00}\wedge B^{00})^{00}=_1A^{00}\wedge B^{00}$.
\item According to (iii), (iv), (ix), (viii) and the assumption we have
\[
A^0\wedge B^0\leq_1(A^0\wedge B^0)^{00}\leq_1(A^{00}\vee B^{00})^0\leq_1A^{000}\wedge B^{000}=A^0\wedge B^0.
\]
\end{enumerate}
\end{proof}

The theorem just proved is very important because we will use it many times in the proofs of most of the following statements.

\begin{remark}\label{rem}
Assume $a,b\in L$ and $a^{00}\vee b^{00}\leq_1(a^0\wedge b^0)^0$. Then because of {\rm(xi)} of Theorem~\ref{th2} we have $(a^{00}\vee b^{00})^0=_1a^0\wedge b^0$. Hence any of the following statements implies the next one:
\begin{align*}
                    1 & \in a^{00}\vee b^{00}, \\
                    1 & \leq_1a^{00}\vee b^{00}, \\
(a^{00}\vee b^{00})^0 & \leq_10, \\
(a^{00}\vee b^{00})^0 & =0, \\
        a^0\wedge b^0 & =0.
\end{align*}
\end{remark}

Applying the previous results, we can show that the set of sharp elements of $\mathbf L=(L,\vee,\wedge,0,1)$ forms a subsemilattice of $(L,\wedge)$.

\begin{proposition}\label{sub}
Let $(L,\vee,\wedge,0,1)$ be a bounded lattice satisfying the {\rm ACC} and $a,b\in S$. Then $a\wedge b\in S$.
\end{proposition}

\begin{proof}
According to (x) of Theorem~\ref{th2} we have $(a\wedge b)^{00}=_1a^{00}\wedge b^{00}=a\wedge b$ and hence $(a\wedge b)^{00}=a\wedge b$ by (vi) of Theorem~\ref{th2}, i.e. $a\wedge b\in S$.
\end{proof}

\begin{remark}
Unfortunately, $S$ need not be a sublattice of $\mathbf L$. Namely, consider the lattice $\mathbf L$ visualized in Fig.~2. Here, obviously, $a,b\in S$ and also $a\wedge b=0\in S$ in accordance with Proposition~\ref{sub}, but $a\vee b=e\notin S$ since $e^{00}=1\neq e$.
\end{remark}

Now we prove an expected statement about the set of dense elements.

\begin{proposition}\label{prop1}
Let $\mathbf L=(L,\vee,\wedge,0,1)$ be a bounded lattice satisfying the {\rm ACC}. Then $D$ forms a $D$-filter of $\mathbf L$.
\end{proposition}

\begin{proof}
Let $a,b\in D$ and $c\in L$. Then according to (x) of Theorem~\ref{th2} we conclude $(a\wedge b)^{00}=_1a^{00}\wedge b^{00}=1\wedge1=1$ and hence $(a\wedge b)^{00}=1$ according to (vi) of Theorem~\ref{th2} proving $a\wedge b\in D$. If $a\leq c$ then $c^0\leq_1a^0=0$ according to (iv) of Theorem~\ref{th2} and hence $c^0=0$, i.e.\ $c\in D$.
\end{proof}

It is well-known that the set $\mathcal F$ of all filters of $\mathbf L$ forms a complete lattice with respect to inclusion. Using Proposition~\ref{prop1} one can recognize that the set of all $D$-filters of $\mathbf L$ forms a complete sublattice of $\mathcal F$ with bottom element $D$.

$D$-filters were described for distributive pseudocomplemented lattices by M.S.~Rao \cite{Ra}. However, we are going to show that similar result can be stated also for lattices being neither pseudocomplemented nor distributive, but satisfying a weaker condition.

\begin{theorem}\label{prop2}
Let $\mathbf L=(L,\vee,\wedge,0,1)$ be a bounded lattice satisfying the {\rm ACC} and $A$ a non-empty subset of $L$ such that for each $x,y\in\overline A$ the following condition holds:
\[
\text{if }1\in(x^{00}\vee z^{00})\wedge(y^{00}\vee z^{00})\text{ for all }z\in A\text{ then }1\in(x^{00}\wedge y^{00})\vee z^{00}\text{ for all }z\in A.
\]
Then $\overline A$ is a $D$-filter of $\mathbf L$.
\end{theorem}

\begin{proof}
Obviously, $D\subseteq\overline A$. Let $a,b\in\overline A$. Then $1\in(a^{00}\vee z^{00})\wedge(b^{00}\vee z^{00})$ for all $z\in A$ and hence $1\in(a^{00}\wedge b^{00})\vee z^{00}=(a\wedge b)^{00}\vee z^{00}$ for all $z\in A$ according to (x) of Theorem~\ref{th2}, i.e.\ $a\wedge b\in\overline A$. Further, if $c\in L$ and $a\leq c$ then $a^{00}\leq_1c^{00}$ according (iv) of Theorem~\ref{th2} and hence $1\in a^{00}\vee z^{00}\leq_1 a^{00}\vee z^{00}$ for all $z\in A$ showing $c\in\overline A$. Altogether, $\overline A$ is a $D$-filter of $\mathbf L$.
\end{proof}

In the next example we show a lattice being neither distributive nor pseudocomplemented, but satisfying the condition from Theorem~\ref{prop2}.

\begin{example}
Consider the non-distributive lattice $\mathbf L$ depicted in Fig.~2:

\vspace*{-4mm}

\begin{center}
\setlength{\unitlength}{1.5mm}
\begin{picture}(40,55)
\put(8,20){\circle*{1.4}}
\put(16,20){\circle*{1.4}}
\put(24,20){\circle*{1.4}}
\put(32,20){\circle*{1.4}}
\put(20,10){\circle*{1.4}}
\put(20,30){\circle*{1.4}}
\put(20,50){\circle*{1.4}}
\put(12,40){\circle*{1.4}}
\put(28,40){\circle*{1.4}}
\put(20,10){\line(-6,5){12}}
\put(20,10){\line(6,5){12}}
\put(20,10){\line(-2,5){4}}
\put(20,10){\line(2,5){4}}
\put(20,30){\line(-6,-5){12}}
\put(20,30){\line(6,-5){12}}
\put(20,30){\line(-2,-5){4}}
\put(20,30){\line(2,-5){4}}
\put(20,30){\line(-4,5){8}}
\put(20,30){\line(4,5){8}}
\put(20,50){\line(-4,-5){8}}
\put(20,50){\line(4,-5){8}}
\put(19.2,6.9){$0$}
\put(5.5,19){$a$}
\put(13.5,19){$b$}
\put(25,19){$c$}
\put(33,19){$d$}
\put(17,29.5){$e$}
\put(9.5,39){$f$}
\put(29,39){$g$}
\put(19.2,51.2){$1$}
\put(17,2){{\rm Fig.~2}}
\put(-10,-2){{\rm Non-distributive non-pseudocomplemented lattice}}
\end{picture}
\end{center}
We have
\[
\begin{array}{r|r|r|r|r|r|r|r|r|r}
          x &   0 &   a &   b &   c &   d &   e &   f &   g &   1 \\
\hline
        x^0 &   1 & bcd & acd & abd & abc &   0 &   0 &   0 &   0 \\
\hline
     x^{00} &   0 &   a &   b &   c &   d &   1 &   1 &   1 &   1 \\
\hline
\overline x & F_e & F_e & F_e & F_e & F_e & F_0 & F_0 & F_0 & F_0
\end{array}
\]
$D=F_e$ and $S=\{0,a,b,c,d,1\}$. Hence $\mathbf L$ is not pseudocomplemented. Further, $F_a$ is a $D$-filter of $\mathbf L$, $\overline{\{a,b\}}=F_e$ is a $D$-filter of $\mathbf L$ in accordance with Theorem~\ref{prop2} and $\{a,b\}^D=\{c,d,e,f,g,1\}$.
\end{example}

\section{Stonean and $D$-Stonean lattices}

The concept of a Stone lattice was introduced by R.~Balbes and A.~Horn \cite{BH}, see also the paper \cite{Sp} by T.P.~Speed. Recall from \cite{BH} that a bounded pseudocomplemented lattice $(L,\vee,\wedge,0,1)$ is called {\em Stone} if
\[
x^*\vee x^{**}=1\text{ and }x^*\vee y^*=(x\wedge y)^*\text{ for all }x,y\in L.
\]
In analogy to this definition we define for our sake the following two concepts.

\begin{definition}\label{def1}
Let $\mathbf L=(L,\vee,\wedge,0,1)$ be a bounded lattice satisfying the {\rm ACC}. Then $\mathbf L$ is called {\em Stonean} if
\begin{equation}\label{equ1}
1\in x^{00}\vee y^{00}\text{ for every }x\in L\text{ and every }y\in x^0
\end{equation}
and {\em $D$-Stonean} if it is both Stonean and if
\begin{equation}\label{equ2}
\text{for all }x,y\in L, x\vee y\in D\text{ is equivalent to }1\in x^{00}\vee y^{00}.
\end{equation}
\end{definition}

Observe that (1) is equivalent to $x^0\subseteq\overline x$ for all $x\in L$, and (2) is equivalent to $\overline x=x^D$ for all $x\in L$. Hence $\mathbf L$ is Stonean if and only if $x^0\subseteq\overline x$ for all $x\in L$, and $\mathbf L$ is $D$-Stonean if $x^0\subseteq\overline x=x^D$ for all $x\in L$. The lattice visualized in Fig.~2 is not Stonean since $b\in a^0$, but $1\notin e=a\vee b=a^{00}\vee b^{00}$.

The following result relates the concept of a Stonean lattice with concepts mentioned before.

\begin{proposition}
Let $\mathbf L=(L,\vee,\wedge,0,1)$ be a bounded lattice satisfying the {\rm ACC} and $a\in L$. Then the following holds:
\begin{enumerate}[{\rm(i)}]
\item If $\mathbf L$ is Stonean and $a^0\leq_2\overline a$ then $\overline{\overline a}=\overline{a^0}$.
\item $\mathbf L$ is Stonean if and only if $\overline{\overline x}\subseteq\overline{x^0}$ for all $x\in L$.
\end{enumerate}
\end{proposition}

\begin{proof}
\
\begin{enumerate}[(i)]
\item Assume that $\mathbf L$ is Stonean and $a^0\leq_2\overline a$. Then $a^0\subseteq\overline a$ and hence $\overline{\overline a}\subseteq\overline{a^0}$. Now let $b\in\overline{a^0}$ and $c\in\overline a$. Since $a^0\leq_2\overline a$ there exists some $d\in a^0$ with $d\leq c$. We conclude $1\in b^{00}\vee d^{00}\leq_1b^{00}\vee c^{00}$ and hence $1\in b^{00}\vee c^{00}$. This shows $b\in\overline{\overline a}$ and hence $\overline{a^0}\subseteq\overline{\overline a}$. Together we obtain $\overline{\overline a}=\overline{a^0}$.
\item According to the observation after Definition~\ref{def1}, $\mathbf L$ is Stonean if and only if $x^0\subseteq\overline x$ for all $x\in L$. Now for every $x\in L$ the inclusion $x^0\subseteq\overline x$ is equivalent to $\overline{\overline x}\subseteq\overline{x^0}$.
\end{enumerate}
\end{proof}

The next result is elementary but it will be used in the sequel.

\begin{lemma}
Let $(L,\vee,\wedge,0,1)$ be a $D$-Stonean lattice and $a\in L$. Then $a\vee a^0\subseteq D$.
\end{lemma}

\begin{proof}
If $b\in a^0$ then $1\in a^{00}\vee b^{00}$ according to (\ref{equ1}) which is equivalent to $a\vee b\in D$ because of (\ref{equ2}).
\end{proof}

\begin{example}
\
\begin{enumerate}[{\rm(i)}]
\item The lattice visualized in Fig.~2 satisfies neither {\rm(\ref{equ1})} nor {\rm(\ref{equ2})} since $b\in a^0$ and $a\vee b=e\in D$, but $1\notin e=a\vee b=a^{00}\vee b^{00}$.
\item Consider the non-distributive lattice $\mathbf L$ depicted in Fig.~3:

\vspace*{-4mm}

\begin{center}
\setlength{\unitlength}{1.5mm}
\begin{picture}(55,45)
\put(10,20){\circle*{1.4}}
\put(10,30){\circle*{1.4}}
\put(20,20){\circle*{1.4}}
\put(30,10){\circle*{1.4}}
\put(30,30){\circle*{1.4}}
\put(30,40){\circle*{1.4}}
\put(40,20){\circle*{1.4}}
\put(50,20){\circle*{1.4}}
\put(50,30){\circle*{1.4}}
\put(30,10){\line(-2,1){20}}
\put(30,10){\line(2,1){20}}
\put(30,10){\line(-1,1){10}}
\put(30,10){\line(1,1){10}}
\put(10,20){\line(0,1){10}}
\put(10,20){\line(2,1){20}}
\put(20,20){\line(1,1){10}}
\put(40,20){\line(-1,1){10}}
\put(50,20){\line(0,1){10}}
\put(50,20){\line(-2,1){20}}
\put(30,30){\line(0,1){10}}
\put(10,30){\line(2,1){20}}
\put(50,30){\line(-2,1){20}}
\put(29.2,6.9){$0$}
\put(7.5,19){$a$}
\put(17.5,19){$b$}
\put(41,19){$c$}
\put(51,19){$d$}
\put(7.5,29){$e$}
\put(30.8,30.5){$f$}
\put(51,29){$g$}
\put(29.2,41.2){$1$}
\put(26,1){{\rm Fig.~3}}
\put(-8,-3){{\rm Non-distributive non-pseudocomplemented non-Stonean lattice}}
\end{picture}
\end{center}
We have
\[
\begin{array}{r|r|r|r|r|r|r|r|r|r}
          x &   0 &      a &      b &      c &      d &      e & f &      g &   1 \\
\hline
        x^0 &   1 &    bcg &    ceg &    beg &    bce &    bcg & 0 &    bce &   0 \\
\hline
     x^{00} &   0 &      e &      b &      c &      g &      e & 1 &      g &   1 \\
\hline
\overline x & F_f & bcdfg1 & adefg1 & adefg1 & abcef1 & bcdfg1 & L & abcef1 & F_0
\end{array}
\]
$D=F_f$ and $S=\{0,b,c,e,g,1\}$. Hence $\mathbf L$ is not pseudocomplemented. Moreover, $\mathbf L$ satisfies neither {\rm(1)} nor {\rm(2)} since $c\in b^0$ and $b\vee c=f\in D$, but $1\notin f=b\vee c=b^{00}\vee c^{00}$. Therefore $\mathbf L$ is not Stonean.
\end{enumerate}
\end{example}

\begin{theorem}
Conditions {\rm(\ref{equ1})} and {\rm(\ref{equ2})} of Definition~\ref{def1} are independent.
\end{theorem}

\begin{proof}
The lattice visualized in Fig.~1 satisfies (\ref{equ1}), but not (\ref{equ2}) since $e\vee g=1\in D$, but $1\notin g=g\vee g=e^{00}\vee g^{00}$. Hence it is Stonean, but not $D$-Stonean. Now consider the non-distributive lattice $\mathbf L$ depicted in Fig.~4:

\vspace*{-4mm}

\begin{center}
\setlength{\unitlength}{1.5mm}
\begin{picture}(55,55)
\put(10,30){\circle*{1.4}}
\put(20,20){\circle*{1.4}}
\put(20,30){\circle*{1.4}}
\put(20,40){\circle*{1.4}}
\put(30,10){\circle*{1.4}}
\put(30,20){\circle*{1.4}}
\put(30,40){\circle*{1.4}}
\put(30,50){\circle*{1.4}}
\put(40,20){\circle*{1.4}}
\put(40,30){\circle*{1.4}}
\put(40,40){\circle*{1.4}}
\put(50,30){\circle*{1.4}}
\put(30,10){\line(-1,1){20}}
\put(30,10){\line(1,1){20}}
\put(30,50){\line(-1,-1){20}}
\put(30,50){\line(1,-1){20}}
\put(20,20){\line(0,1){20}}
\put(40,20){\line(0,1){20}}
\put(30,10){\line(0,1){10}}
\put(30,40){\line(0,1){10}}
\put(30,20){\line(-1,1){10}}
\put(30,20){\line(1,1){10}}
\put(30,40){\line(-1,-1){10}}
\put(30,40){\line(1,-1){10}}
\put(29.2,6.9){$0$}
\put(17.5,19){$a$}
\put(31,19){$b$}
\put(41,19){$c$}
\put(7.5,29){$d$}
\put(17.5,29){$e$}
\put(41,29){$f$}
\put(51,29){$g$}
\put(17.5,40){$h$}
\put(31,40){$i$}
\put(41,40){$j$}
\put(29.2,51.2){$1$}
\put(26,1){{\rm Fig.~4}}
\put(-8,-3){{\rm Non-distributive non-pseudocomplemented non-Stonean lattice}}
\end{picture}
\end{center}
We have
\[
\begin{array}{r|r|r|r|r|r|r|r|r|r|r|r|r}
          x &   0 &   a &   b &   c &   d &   e &   f &   g &   h &   i &   j & 1 \\
\hline
        x^0 &   1 &   j &  dg &   h &   j &   g &   d &   h &   g &   0 &   d & 0 \\
\hline
     x^{00} &   0 &   d &   b &   g &   d &   h &   j &   g &   h &   1 &   j & 1 \\
\hline
\overline x & F_i & F_c & F_i & F_a & F_c & F_c & F_a & F_a & F_c & F_0 & F_a & F_0
\end{array}
\]
$D=F_i$ and $S=\{0,b,d,g,h,j,1\}$. Hence $\mathbf L$ is not pseudocomplemented. The lattice $\mathbf L$ satisfies (\ref{equ2}), but not (\ref{equ1}) since $d\in b^0$, but $1\notin h=b\vee d=b^{00}\vee d^{00}$. Therefore it is not Stonean.
\end{proof}

How the concept of a Stonean lattice is related with its filters is shown in the following theorem.

\begin{theorem}\label{thm9}
Let $\mathbf L=(L,\vee,\wedge,0,1)$ be a Stonean lattice satisfying the {\rm ACC}. Then {\rm(\ref{equ2})} is equivalent to any single of the following statements:
\begin{enumerate}[{\rm(i)}]
\item $\overline F=F^D$ for all filters $F$ of $\mathbf L$,
\item $\overline x=x^D$ for all $x\in L$,
\item for every two filters $F,G$ of $\mathbf L$, $F\cap G\subseteq D$ is equivalent to $F\subseteq\overline G$.
\end{enumerate}
\end{theorem}

\begin{proof}
$\text{}$ \\
(\ref{equ2}) $\Rightarrow$ (i): \\
This follows from
\begin{align*}
\overline F & =\{x\in L\mid1\in x^{00}\vee y^{00}\text{ for all }y\in F\}, \\
        F^D & =\{x\in L\mid x\vee y\in D\text{ for all }y\in F\}
\end{align*}
for all filters $F$ of $\mathbf L$. \\
(i) $\Rightarrow$ (ii): \\
For all $x\in L$ we have $\overline x=\overline{F_x}=F_x^D=x^D$. \\
(ii) $\Rightarrow$ (iii): \\
Let $F,G$ be filters of $\mathbf L$. If $a\in F\cap G$ then $a=a\vee a\in F\vee G$. If, conversely, $a\in F\vee G$ then there exists some $b\in F$ and some $c\in G$ with $b\vee c=a$ and hence $a\in F\cap G$. This shows $F\cap G=F\vee G$. Now the following are equivalent:
\begin{align*}
F\cap G & \subseteq D, \\
F\vee G & \subseteq D, \\
x\vee y & \in D\text{ for all }x\in F\text{ and all }y\in G, \\
      y & \in x^D\text{ for all }x\in F\text{ and all }y\in G, \\
      y & \in\overline x\text{ for all }x\in F\text{ and all }y\in G, \\
      1 & \in x^{00}\vee y^{00}\text{ for all }x\in F\text{ and all }y\in G, \\
      F & \subseteq\overline G.
\end{align*}
(iii) $\Rightarrow$ (\ref{equ2}): \\
For all $x,y\in L$ the following are equivalent: $x\vee y\in D$; $F_{x\vee y}\subseteq D$; $F_x\cap F_y\subseteq D$; $F_x\subseteq\overline{F_y}$; $x\in\overline y$; $1\in x^{00}\vee y^{00}$.
\end{proof}

The results of the previous theorem can be checked in the following example.

\begin{example}\label{exam2b}
Consider the non-distributive lattice $\mathbf L$ depicted in Fig.~5:

\vspace*{-4mm}

\begin{center}
\setlength{\unitlength}{1.5mm}
\begin{picture}(45,45)
\put(10,20){\circle*{1.4}}
\put(10,30){\circle*{1.4}}
\put(20,10){\circle*{1.4}}
\put(20,25){\circle*{1.4}}
\put(20,40){\circle*{1.4}}
\put(35,25){\circle*{1.4}}
\put(20,10){\line(-1,1){10}}
\put(20,10){\line(0,1){30}}
\put(20,10){\line(1,1){15}}
\put(10,20){\line(0,1){10}}
\put(20,40){\line(-1,-1){10}}
\put(20,40){\line(1,-1){15}}
\put(19.2,6.9){$0$}
\put(7.5,19){$a$}
\put(21,24){$b$}
\put(36,24){$c$}
\put(7.5,29){$d$}
\put(19.2,41.2){$1$}
\put(17,2){{\rm Fig.~5}}
\put(-14,-2){{\rm Non-distributive non-pseudocomplemented $D$-Stonean lattice}}
\end{picture}
\end{center}
We have
\[
\begin{array}{r|r|r|r|r|r|r}
                                   x &   0 &   a &    b &    c &   d &   1 \\
\hline
                                 x^0 &   1 &  bc &   cd &   bd &  bc &   0 \\
\hline
                              x^{00} &   0 &   d &    b &    c &   d &   1 \\
\hline
\overline{F_x}=F_x^D=\overline x=x^D & F_1 & bc1 & acd1 & abd1 & bc1 & F_0
\end{array}
\]
$D=F_1$ and $S=\{0,b,c,d,1\}$. Hence $\mathbf L$ is not pseudocomplemented, but it is $D$-Stonean. Moreover, for $x,y\in L$ both $F_x\cap F_y\subseteq D$ and $F_x\subseteq\overline{F_y}$ are equivalent to
\[
1\in\{x,y\}\text{ or }(x,y\in\{a,b,c,d\}\text{ and }x\neq y\text{ and }\{x,y\}\neq\{a,d\})
\]
in accordance with Theorem~\ref{thm9}.
\end{example}

\section{Coherent and closed filters}

We are going to introduce the so-called coherent filters. For this, let us define the operator $c$ on filters of $\mathbf L$ as follows:
\[
c(F):=\{x\in L\mid\overline x\wedge F=L\}.
\]
Now we define

\begin{definition}
A {\em filter} $F$ of a bounded lattice $\mathbf L=(L,\vee,\wedge,0,1)$ satisfying the {\rm ACC} is called {\em coherent} if $c(F)=F$.
\end{definition}

One can easily show that if $F$ and $G$ are filters of $\mathbf L$ with $F\subseteq G$ then $c(F)\subseteq c(G)$. Hence $c(F\cap G)\subseteq c(F)\cap c(G)$ for all filters $F,G$ of $\mathbf L$.

\begin{example}
Consider the lattice from Fig.~2. Then we have
\[
\begin{array}{r|r|r|r|r|r|r|r|r|r}
     x &   0 &   a &   b &   c &   d &   e &   f &   g &   1 \\
\hline
c(F_x) & F_0 & F_e & F_e & F_e & F_e & F_e & F_e & F_e & F_e
\end{array}
\]
and hence $F_0$ and $F_e$ are the only coherent filters. The filters $F_a$ and $F_f$ of the lattice depicted in Fig.~3 are coherent, but the filter $F_b$ is not since $b\in F_b\setminus c(F_b)$.
\end{example}

\begin{lemma}\label{lem2}
If $\mathbf L=(L,\vee,\wedge,0,1)$ is a bounded lattice satisfying the {\rm ACC} and $F$ a filter of $\mathbf L$ satisfying $\overline x\cup F=L$ for all $x\in F$ then $F\subseteq c(F)$.
\end{lemma}

\begin{proof}
Let $a\in F$ and $b\in L$. Then $1\in F\cap\overline a$ and either $b\in\overline a$ and hence $b=b\wedge1\in\overline a\wedge F$ or $b\in F$ and hence $b=1\wedge b\in\overline a\wedge F$. This shows $a\in c(F)$ and therefore $F\subseteq c(F)$.
\end{proof}

Since $\overline x=L$ for all $x\in D$ we have $F\subseteq c(F)$ for all filters $F$ of $\mathbf L$ being contained in $D$.

It is easy to see that if a lattice $\mathbf L$ is distributive then $c(F)$ is closed under $\wedge$ for every filter $F$ of $\mathbf L$. However, we can show that this holds for not necessarily distributive lattices under a weak condition.

\begin{proposition}\label{prop3}
Let $\mathbf L=(L,\vee,\wedge,0,1)$ be a bounded lattice satisfying the {\rm ACC} and $F$ a proper filter of $\mathbf L$ and $a,b\in L$. Then the following holds:
\begin{enumerate}[{\rm(i)}]
\item If $c(F)$ is closed with respect to $\wedge$ then $c(F)$ is a $D$-filter of $F$,
\item the inclusion $\overline a\wedge\overline b\subseteq\overline{a\wedge b}$ holds if and only if $x,y\in L$ and $1\in(x^{00}\vee a^{00})\wedge(y^{00}\vee b^{00})$ together imply $1\in(x\wedge y)^{00}\vee(a\wedge b)^{00}$,
\item if $\overline x\wedge\overline y\subseteq\overline{x\wedge y}$ for all $x,y\in c(F)$ then $c(F)$ is closed with respect to $\wedge$ and hence a $D$-filter of $\mathbf L$.
\end{enumerate}
\end{proposition}

\begin{proof}
\
\begin{enumerate}[(i)]
\item If $a\in D$ then $a^{00}=1$ and hence $\overline a=L$ whence $L=L\wedge1\subseteq L\wedge F=\overline a\wedge F\subseteq L$, i.e.\ $\overline a\wedge F=L$ showing $a\in c(F)$. Therefore $D\subseteq c(F)$. If $b\in c(F)$, $c\in L$ and $b\leq c$ then $b^{00}\leq_1c^{00}$ and hence $L=\overline b\wedge F\subseteq\overline c\wedge F\subseteq L$ which implies $\overline c\wedge F=L$, i.e.\ $c\in c(F)$.
\item We have
\begin{align*}
\overline a\wedge\overline b & =\{x\wedge y\mid x,y\in L\text{ and }1\in(x^{00}\vee a^{00})\wedge(y^{00}\vee b^{00})\}, \\
        \overline{a\wedge b} & =\{x\in L\mid1\in x^{00}\vee(a\wedge b)^{00}\}.
\end{align*}
\item if $\overline x\wedge\overline y\subseteq\overline{x\wedge y}$ for all $x,y\in c(F)$ then for all $x,y\in c(F)$ we have
\[
L=L\wedge L=(\overline x\wedge F)\wedge(\overline y\wedge F)=(\overline x\wedge\overline y)\wedge(F\wedge F)=(\overline x\wedge\overline y)\wedge F\subseteq\overline{x\wedge y}\wedge F\subseteq L
\]
and hence $\overline{x\wedge y}\wedge F=L$ which means nothing else than $x\wedge y\in c(F)$.
\end{enumerate}
\end{proof}

The condition in {\rm(ii)} of Proposition~\ref{prop3} holds for all proper filters of the lattice from Fig.~2.

For $D$-Stonean lattices we can prove the following result.

\begin{theorem}\label{th1}
Let $\mathbf L=(L,\vee,\wedge,0,1)$ be a $D$-Stonean lattice and $F$ a $D$-filter of $\mathbf L$. Then $c(F)\subseteq F$.
\end{theorem}

\begin{proof}
Let $a\in c(F)$. Then $\overline a\wedge F=L$ and hence there exists some $b\in\overline a$ and some $c\in F$ with $b\wedge c=a$. Now
\[
\overline a=\{x\in L\mid1\in x^{00}\vee a^{00}\}=\{x\in L\mid x\vee a\in D\}
\]
and hence $b=b\vee(b\wedge c)=b\vee a\in D\subseteq F$ which implies $a=b\wedge c\in F$.
\end{proof}

Let us note that the condition that $\mathbf L$ is $D$-Stonean is only sufficient but not necessary.

\begin{example}
The filter $F_a$ of the non-Stonean lattice visualized in Fig.~2 is not coherent since $a\in F_a\setminus c(F_a)$, but $c(F_a)=F_e\subseteq F_a$.
\end{example}

Combining Lemma~\ref{lem2} and Theorem~\ref{th1} we obtain

\begin{corollary}
If $\mathbf L=(L,\vee,\wedge,0,1)$ is a $D$-Stonean lattice and $F$ a $D$-filter of $\mathbf L$ satisfying $\overline x\cup F=L$ for all $x\in F$ then $F$ is coherent.
\end{corollary}

Now we turn our attention to the so-called closed filters. A {\em filter} $F$ of $\mathbf L$ is called {\em closed} if it is a closed subset of $L$ as defined in the introduction, i.e.\ if $\overline{\overline F}=F$. Of course, $F\subseteq\overline{\overline F}$ holds for every filter $F$ of $\mathbf L$.

\begin{example}
The filter $F_a$ of the $D$-Stonean lattice visualized in Fig.~5 is closed and coherent since $\overline{\overline{F_a}}=\overline{\{b,c,1\}}=F_a$ and $c(F_a)=F_a$.
\end{example}

\begin{corollary}
Let $\mathbf L=(L,\vee,\wedge,0,1)$ be a bounded lattice satisfying the {\rm ACC} and $A\subseteq L$. Then the following holds:
\begin{enumerate}[{\rm(i)}]
\item $D\subseteq\overline A$,
\item $\overline A=L$ if and only if $A\subseteq D$,
\item every closed filter of $\mathbf L$ is a $D$-filter.
\end{enumerate}
\end{corollary}

\begin{proof}
\
\begin{enumerate}
\item[(i)] and (ii) We use Lemma~\ref{lem4}.
\item[(iii)] This follows from (i).
\end{enumerate}
\end{proof}

\begin{example}
The filter $F_e$ of the lattice depicted in Fig.~2 is both coherent and closed since $c(F_e)=F_e$ and $\overline{\overline{F_e}}=\overline{F_0}=F_e$, but its subfilter $F_f$ is neither coherent nor closed since $e\in c(F_f)\setminus F_f$ and $\overline{\overline{F_f}}=\overline{F_0}=F_e\neq F_f$.
\end{example}

\begin{remark}
Since the set of closed subsets of $L$ is closed under arbitrary intersection, the same is true for the set of closed filters and the latter forms a complete lattice with respect to inclusion. Moreover, for every filter $F$ of $\mathbf L$ there is a smallest closed filter of $\mathbf L$ including it.
\end{remark}

\section{Maximal, prime and median filters}

First let us recall some well-known classes of lattice filters. We say that a {\em filter} $F$ of a lattice $(L,\vee,\wedge)$ is
\begin{itemize}
\item {\em proper} if $F\neq L$,
\item {\em maximal} if it is a maximal proper filter,
\item {\em prime} if $x,y\in L$ and $x\vee y\in F$ imply $x\in F$ or $y\in F$.
\end{itemize}
It is well-known (see e.g.\ \cite{Sz}) that every maximal filter of a distributive lattice is prime. Unfortunately, this does not hold for non-distributive lattices. For example, the filter $F_a$ of the lattice in Fig.~5 is maximal, but it is not prime since $b\vee c=1\in F_a$, but $b,c\notin F_a$. However, we can prove the following:

\begin{theorem}\label{th3}
Let $\mathbf L=(L,\vee,\wedge,0,1)$ be a bounded lattice satisfying the {\rm ACC} and $F$ a proper filter of $\mathbf L$. Then the following holds:
\begin{enumerate}[{\rm(i)}]
\item $F$ is maximal if and only if $x^0\cap F\neq\emptyset$ for all $x\in L\setminus F$,
\item if $F$ is maximal then $F$ is a $D$-filter,
\item if $F$ is maximal and $(x\vee y)\wedge z=0$ for all $x,y\in L\setminus F$ and all $z\in F$ with $x\wedge z=y\wedge z=0$ then $F$ is prime.
\end{enumerate}
\end{theorem}

\begin{proof}
\
\begin{enumerate}[(i)]
\item Assume $F$ to be maximal and $a\in L\setminus F$. Then $G:=\{x\in L\mid\text{there exists some }f\in F\text{ with }f\wedge a\leq x\}$ is a filter of $\mathbf L$ strictly including $F$. Since $F$ is maximal we conclude $G=L$. Hence $0\in G$, i.e.\ there exists some $g\in F$ with $g\wedge a\leq0$, it means $g\wedge a=0$. Therefore there exists some $b\in a^0$ with $g\leq b$. Since $F$ is a filter we conclude $b\in F$ and hence $b\in a^0\cap F$ whence $a^0\cap F\neq\emptyset$. Conversely, assume $x^0\cap F\neq\emptyset$ for all $x\in L\setminus F$. Suppose $F$ not to be maximal. Then there exists a proper filter $H$ of $\mathbf L$ strictly including $F$. Let $c\in H\setminus F$. Then $c^0\cap F\neq\emptyset$. Let $d\in c^0\cap F$. Then $c,d\in H$. Since $H$ is a filter of $\mathbf L$ we have $0=c\wedge d\in H$ and hence $H=L$, a contradiction. This shows that $F$ is maximal.
\item Assume $F$ to be maximal, but not being a $D$-filter. Then there exists some $a\in D\setminus F$. According to (i) $0\cap F=a^0\cap F\neq\emptyset$ and hence $0\in F$ which implies $F=L$, a contradiction. Therefore $F$ is a $D$-filter.
\item Assume $F$ to be maximal and $(x\vee y)\wedge z=0$ for all $x,y\in L\setminus F$ and all $z\in F$ with $x\wedge z=y\wedge z=0$. Suppose $F$ to be not prime. Then there exist some $a,b\in L\setminus F$ with $a\vee b\in F$. According to (i) we have $a^0\cap F,b^0\cap F\neq\emptyset$. Let $f\in a^0\cap F$ and $g\in b^0\cap F$ and put $h:=f\wedge g$. Then $h\in F$ and $a\wedge f=b\wedge g=0$ and hence $a\wedge h=b\wedge h=0$. By the above assumption we conclude $0=(a\vee b)\wedge h\in F$ which implies that $F$ is not proper which is a contradiction. Hence $F$ is prime. 
\end{enumerate}
\end{proof}

Of course, condition (iii) of Theorem~\ref{th3} holds for any distributive lattice. However, the following example shows that it may be satisfied also in a non-distributive lattice.

\begin{example}
The non-distributive lattice visualized in Fig.~6

\vspace*{-4mm}

\begin{center}
\setlength{\unitlength}{1.5mm}
\begin{picture}(35,45)
\put(10,30){\circle*{1.4}}
\put(20,10){\circle*{1.4}}
\put(20,20){\circle*{1.4}}
\put(20,30){\circle*{1.4}}
\put(20,40){\circle*{1.4}}
\put(30,30){\circle*{1.4}}
\put(20,10){\line(0,1){30}}
\put(20,20){\line(-1,1){10}}
\put(20,20){\line(1,1){10}}
\put(20,40){\line(-1,-1){10}}
\put(20,40){\line(1,-1){10}}
\put(19.2,6.9){$0$}
\put(21,19){$a$}
\put(7.5,29){$b$}
\put(21,29){$c$}
\put(31,29){$d$}
\put(19.2,41.2){$1$}
\put(16,1){{\rm Fig.~6}}
\put(6,-3){{\rm Non-distributive lattice}}
\end{picture}

\vspace*{2mm}

\end{center}
satisfies the condition in {\rm(iii)} of Theorem~\ref{th3} for the maximal proper filter $F_a$. In accordance with this theorem, this filter is prime.
\end{example}

In the following we adopt a certain modification of the concept of a median filter defined in \cite{Ra}.

\begin{definition}
Let $\mathbf L=(L,\vee,\wedge,0,1)$ be a bounded lattice satisfying the {\rm ACC}. The filter $F$ of $\mathbf L$ is called {\em median} if it is maximal and if for each $x\in F$ there exists some $y\in L\setminus F$ with $1\in x^{00}\vee y^{00}$.
\end{definition}

\begin{example}
The filter $F_a$ of the non-Stonean lattice $(L,\vee,\wedge,0,1)$ visualized in Fig.~3 is median, closed and coherent since $b\in L\setminus F_a$, $1\in x^{00}\vee b^{00}$ for every $x\in F_a$ and $\overline{\overline{F_a}}=\overline{\{b,c,d,f,g,1\}}=F_a$. The filter $F_a$ of the non-Stonean lattice $(L,\vee,\wedge,0,1)$ depicted in Fig.~4 is prime, median, closed and coherent since $c\in L\setminus F_a$, $1\in x^{00}\vee c^{00}$ for all $x\in F_a$ and $\overline{\overline{F_a}}=\overline{F_c}=F_a$. We list all closed filters of the lattices from Fig.~1 to Fig.~5:
\[
\begin{array}{c|l}
\text{Fig.} & \text{closed filters} \\
\hline
          1 & F_0,F_d,F_f,F_1 \\
\hline
          2 & F_0,F_e \\
\hline
          3 & F_0,F_a,F_d,F_f \\
\hline
          4 & F_0,F_a,F_c,F_i \\
\hline
          5 & F_0,F_b,F_c,F_d,F_1
\end{array}
\]
\end{example}

Now we present several basic properties of proper, prime and median filters.

\begin{theorem}
Let $\mathbf L=(L,\vee,\wedge,0,1)$ be a bounded lattice satisfying the {\rm ACC}, $F$ a proper filter of $\mathbf L$ and $a\in L$. Then the following holds:
\begin{enumerate}[{\rm(i)}]
\item If $x^{00}\vee y^{00}\leq_1(x\vee y)^{00}$ for all $x,y\in L$, $F$ is a prime $D$-filter and $a\in L\setminus F$ then ${\overline a}\subseteq F$,
\item if $x^{00}\vee y^{00}\leq_1(x\vee y)^{00}$ for all $x,y\in L$, $F$ is a median $D$-filter and $a\in F$ then $\overline{\overline a}\subseteq F$,
\item If $a\in F$ then $a^0\not\subseteq F$,
\item if $\mathbf L$ is $D$-Stonean and $F$ a prime $D$-filter of $\mathbf L$ then $a\in F$ if and only if $a^0\not\subseteq F$.
\end{enumerate}
\end{theorem}

\begin{proof}
\
\begin{enumerate}[(i)]
\item Assume $x^{00}\vee y^{00}\leq_1(x\vee y)^{00}$ for all $x,y\in L$ and let $F$ be a prime $D$-filter of $\mathbf L$, $a\in L\setminus F$ and $b\in{\overline a}$. Then $1\in a^{00}\vee b^{00}\leq_1(a\vee b)^{00}$. Since $(a\vee b)^{00}$ is an antichain we have $(a\vee b)^{00}=1$ which implies $(a\vee b)^0=(a\vee b)^{000}=0$, i.e.\ $a\vee b\in D$. Since $F$ is a $D$-filter, we have $D\subseteq F$ and therefore $a\vee b\in F$. Because $a\notin F$ and $F$ is prime we conclude $b\in F$ proving ${\overline a}\subseteq F$.
\item Assume $x^{00}\vee y^{00}\leq_1(x\vee y)^{00}$ for all $x,y\in L$ and let $F$ be a median $D$-filter of $\mathbf L$, $a\in F$ and $b\in\overline{\overline a}$. Since $F$ is median there exists some $c\in L\setminus F$ with $1\in a^{00}\vee c^{00}$. Hence $c\in{\overline a}=\overline{\overline{\overline a}}\subseteq{\overline b}$. Since $c\notin F$ we have ${\overline c}\subseteq F$ according to (i) and we obtain $b\in\overline{\overline b}\subseteq{\overline c}\subseteq F$ proving $\overline{\overline a}\subseteq F$.
\item If we would have $a^0\subseteq F$ for some $a\in F$ then $0=a\wedge a^0\subseteq F$ and hence $F=L$ contradicting the assumption that $F$ is proper.
\item Suppose $\mathbf L$ to be $D$-Stonean, $F$ to be a prime $D$-filter of $\mathbf L$ and $a^0\not\subseteq F$. Then there exists some $b\in a^0\setminus F$. Since $\mathbf L$ is $D$-Stonean we have $1\in a^{00}\vee b^{00}$ and hence $a\vee b\in D\subseteq F$ according to the assumption that $F$ is a $D$-filter. Because $F$ is prime and $b\notin F$ we obtain $a\in F$. The converse implication follows from (i).
\end{enumerate}
\end{proof}

The next result illuminates the relationship between coherent and median filters.

\begin{theorem}
Let $\mathbf L=(L,\vee,\wedge,0,1)$ be a bounded lattice satisfying the {\rm ACC} and $F$ a maximal filter of $\mathbf L$. Then the following holds:
\begin{enumerate}[{\rm(i)}]
\item If $F$ is coherent then it is median,
\item if $\overline F\not\subseteq F$ then $F$ is median,
\item if $F=\overline{L\setminus F}$ then $F$ is median.
\end{enumerate}
\end{theorem}

\begin{proof}
Let $a\in F$.
\begin{enumerate}[(i)]
\item Assume $F$ to be coherent. Then $a\in c(F)$, i.e.\ $\overline a\wedge F=L$ and hence there exists some $b\in\overline a$ and some $c\in F$ with $b\wedge c=0$. Because of $b\in\overline a$ we have $1\in a^{00}\vee b^{00}$. Now $b\in F$ would imply $0=b\wedge c\in F$ and hence $F=L$ contradicting the maximality of $F$. Hence $b\notin F$ showing $F$ to be median.
\item Assume $\overline F\not\subseteq F$. Then there exists some $b\in\overline F\setminus F$. Hence $1\in a^{00}\vee b^{00}$ and $b\in L\setminus F$ proving $F$ to be median.
\item Since $a\in\overline{L\setminus F}$ we have $1\in a^{00}\vee x^{00}$ for all $x\in L\setminus F$. Because $F$ is proper there exists such an $x$.
\end{enumerate}
\end{proof}

An interesting property of filters being both median and prime shown for distributive pseudocomplemented lattices by M.S.~Rao \cite{Ra} can be proved also in a more general case.

\begin{proposition}
Let $\mathbf L=(L,\vee,\wedge,0,1)$ be a $D$-Stonean lattice, $F$ a median prime filter of $\mathbf L$ and $a,b\in L$. Then the following holds:
\begin{enumerate}[{\rm(i)}]
\item If $F$ is a $D$-filter of $\mathbf L$, $a\in F$ and $\overline a=\overline b$ then $b\in F$,
\item if $a\vee b\in F$ then there exists some $c\in L\setminus F$ such that $\{a\vee c,b\vee c\}\cap D\neq\emptyset$.
\end{enumerate}
\end{proposition}

\begin{proof}
\
\begin{enumerate}[(i)]
\item Assume $F$ to be a $D$-filter of $\mathbf L$, $a\in F$ and $\overline a=\overline b$. Since $F$ is median there exists some $c\in L\setminus F$ with $1\in a^{00}\vee c^{00}$. Hence $c\in\overline a=\overline b$, i.e.\ $1\in b^{00}\vee c^{00}$. Since $\mathbf L$ is $D$-Stonean and $F$ is a $D$-filter we have $b\vee c\in D\subseteq F$. Because $F$ is prime we conclude $b\in F$.
\item Suppose $a\vee b\in F$. Since $F$ is prime we have $a\in F$ or $b\in F$. Because $F$ is median, in the first case we see that there exists some $c\in L\setminus F$ with $1\in a^{00}\vee c^{00}$. Since $\mathbf L$ is $D$-Stonean we obtain $a\vee c\in D$. The case $b\in F$ can be treated analogously.
\end{enumerate}
\end{proof}

{\bf Data availability statement} No datasets were generated or analyzed during the current study.

Authors' addresses:

Ivan Chajda \\
Palack\'y University Olomouc \\
Faculty of Science \\
Department of Algebra and Geometry \\
17.\ listopadu 12 \\
771 46 Olomouc \\
Czech Republic \\
ivan.chajda@upol.cz

Miroslav Kola\v r\'ik \\
Palack\'y University Olomouc \\
Faculty of Science \\
Department of Computer Science \\
17.\ listopadu 12 \\
771 46 Olomouc \\
Czech Republic \\
miroslav.kolarik@upol.cz

Helmut L\"anger \\
TU Wien \\
Faculty of Mathematics and Geoinformation \\
Institute of Discrete Mathematics and Geometry \\
Wiedner Hauptstra\ss e 8-10 \\
1040 Vienna \\
Austria, and \\
Palack\'y University Olomouc \\
Faculty of Science \\
Department of Algebra and Geometry \\
17.\ listopadu 12 \\
771 46 Olomouc \\
Czech Republic \\
helmut.laenger@tuwien.ac.at

\end{document}